\documentclass[11pt]{amsart}

\usepackage{amsmath,amssymb,amsthm,mathtools}
\usepackage[colorlinks=true,linkcolor=blue,citecolor=blue,urlcolor=blue]{hyperref}

\theoremstyle{plain}
\newtheorem{theorem}{Theorem}
\newtheorem{lemma}{Lemma}
\newtheorem{proposition}{Proposition}

\theoremstyle{definition}
\newtheorem{definition}{Definition}

\theoremstyle{remark}
\newtheorem{remark}{Remark}

\numberwithin{equation}{section}

\newcommand{\R}{\mathbb R}
\newcommand{\N}{\mathbb N}
\newcommand{\E}{\mathbb E}
\newcommand{\1}{\mathbf 1}
\DeclareMathOperator{\sgn}{sgn}

\title[What to Expect When You're Expecting]{What to Expect When You're Expecting}
\author{Mark Whitmeyer}
\date{\today}
\thanks{\emph{Acknowledgements.} Dedicated to KS. I thank Joseph Whitmeyer for his comments. I used ChatGPT as one would an RA and solicited feedback from \href{refine.ink}{refine.ink}.}

\begin{document}

\begin{abstract}
The marginal degree of sums in dimension \(n\) is the smallest integer \(k\) such that the joint distributions of all subcollections of at most \(k\) coordinates of a real-valued random vector \(\left(X_1,\ldots,X_n\right)\) determine the value of
\(\E\left(X_1+\cdots+X_n\right)\), whenever this expectation is defined. For every \(n\ge2\), we prove that this marginal degree is
\(\left\lceil n/2\right\rceil\). The upper bound follows from a theorem of Simons \cite{Simons1977Unexpected}. The lower bound is proved by constructing, for every
\(1\le k<\left\lceil n/2\right\rceil\), two joint laws whose marginals of dimension at most \(k\) agree, but for which the corresponding expectations of \(X_1+\cdots+X_n\) are defined and unequal.
\end{abstract}

\maketitle

\section{The marginal degree of sums}
Simons \cite{Simons1977Unexpected} revealed that the obviously true fact that the value of the expectation \(\E\left[X+Y+Z\right]\) depends on \(X\), \(Y\), and \(Z\) only through their marginal distributions is actually false, although this paradoxical behavior cannot manifest for \(\E\left[X+Y\right]\). Simons concludes his paper by noting that \(\E\left[X+Y+Z\right]\) depends on \(X\), \(Y\), and \(Z\) only through the marginal distribution of \(Z\) and the joint  distribution of \(X\), \(Y\); and similarly for four random variables. This begs the question corresponding to the abstract: what is the smallest integer \(k\) such that the joint distributions of all subcollections of at most \(k\) coordinates of a real-valued random vector \(\left(X_1,\ldots,X_n\right)\) determine the value of
\(\E\left(X_1+\cdots+X_n\right)\), whenever this expectation is defined?

Throughout, \(\N\coloneqq\left\{0,1,2,\ldots\right\}\). For a real-valued random variable \(W\), we say that the extended expectation \(\E W\) is defined if at least one of \(\E W^+\) and \(\E W^-\) is finite.\footnote{Recall that \(W^+\coloneqq\max\left\{W,0\right\}\) and \(W^-\coloneqq\max\left\{-W,0\right\}\).} In that case \(\E W\coloneqq\E W^+-\E W^-\), with the usual convention allowing the values \(\pm\infty\).

\begin{definition}
Let \(n\ge2\). For \(1\le k\le n\), say that the extended expectation of the sum is determined by \(k\)-marginals if the following holds: whenever \(X=\left(X_1,\ldots,X_n\right)\) and \(Y=\left(Y_1,\ldots,Y_n\right)\) are random vectors in \(\R^n\) such that \(\left(X_i\right)_{i\in S}\stackrel{d}=\left(Y_i\right)_{i\in S}\) for every \(S\subseteq\left[n\right]\) with \(\lvert S\rvert\le k\), and both \(\E\sum_{i=1}^nX_i\) and \(\E\sum_{i=1}^nY_i\) are defined, then
\[
\E\sum_{i=1}^nX_i=\E\sum_{i=1}^nY_i.
\]
We term the least such \(k\) the marginal degree of sums in dimension \(n\) and denote it by \(\sigma\left(n\right)\).
\end{definition}

Our main theorem is
\begin{theorem}
For every \(n\ge2\), \(\sigma\left(n\right)=\left\lceil\frac{n}{2}\right\rceil\).
\end{theorem}

The proof has two parts. First we prove the upper bound, which is a direct consequence of Simons's two-variable theorem \cite{Simons1977Unexpected}. Then we prove the lower bound by constructing, whenever \(k<\left\lceil n/2\right\rceil\), two joint distributions whose marginals of dimension at most \(k\) agree, but for which the expectations of the sums are defined and unequal.

\subsection{The upper bound}

We first recall the two-variable theorem of Simons.

\begin{lemma}[Simons \cite{Simons1977Unexpected}]
Let \(X,Y\) be real random variables. If \(\E\left(X+Y\right)\) is defined, then its value depends only on the marginal laws of \(X\) and \(Y\).
\end{lemma}

\begin{proof}
This is Simons's theorem \cite[Theorem, p.~157]{Simons1977Unexpected}; for completeness we include the short proof. Equivalently, we must show that if \(\left(X,Y\right)\) and \(\left(X',Y'\right)\) have the same one-dimensional marginals and both \(\E\left(X+Y\right)\) and \(\E\left(X'+Y'\right)\) are defined, then \(\E\left(X+Y\right)=\E\left(X'+Y'\right)\).

For \(c>0\), let \(W^{\left(c\right)}\coloneqq\max\left\{-c,\min\left\{W,c\right\}\right\}\) be the truncation of a random variable \(W\) to the interval \(\left[-c,c\right]\). The truncated variables are bounded, so the linearity of expectation for bounded random variables yields \(\E\left(X^{\left(c\right)}+Y^{\left(c\right)}\right)=\E X^{\left(c\right)}+\E Y^{\left(c\right)}\). Since \(\left(X,Y\right)\) and \(\left(X',Y'\right)\) have the same one-dimensional marginals, for every \(c>0\),
\[
\E X^{\left(c\right)}+\E Y^{\left(c\right)}=\E {X'}^{\left(c\right)}+\E {Y'}^{\left(c\right)}=\E\left({X'}^{\left(c\right)}+{Y'}^{\left(c\right)}\right).
\]

It remains to pass to the limit as \(c\to\infty\). We first note the following elementary fact. For fixed real numbers \(x,y\), let
\[
x_c\coloneqq\max\left\{-c,\min\left\{x,c\right\}\right\}
\qquad \text{and} \qquad
y_c\coloneqq\max\left\{-c,\min\left\{y,c\right\}\right\}.
\]
Then,
\[
\left(x_c+y_c\right)^+\uparrow\left(x+y\right)^+
\qquad \text{and} \qquad
\left(x_c+y_c\right)^-\uparrow\left(x+y\right)^-.
\]
Indeed, this is immediate when \(x\) and \(y\) have the same sign. If
\(x\ge0\ge y\), write \(a\coloneqq x\) and \(b\coloneqq-y\). Then
\(x_c+y_c=\min\left\{a,c\right\}-\min\left\{b,c\right\}\), from which the claim follows according as \(a\ge b\) or \(a\le b\). The case \(y\ge0\ge x\) is the same.

Applying this observation with \(x=X\left(\omega\right)\) and
\(y=Y\left(\omega\right)\), and then using the monotone convergence theorem, we get
\[
\E\left[\left(X^{\left(c\right)}+Y^{\left(c\right)}\right)^+\right]
\to
\E\left(X+Y\right)^+
\quad \text{and} \quad
\E\left[\left(X^{\left(c\right)}+Y^{\left(c\right)}\right)^-\right]
\to
\E\left(X+Y\right)^-.
\]
Since \(\E\left(X+Y\right)\) is defined, these two limiting quantities are not
both infinite. Hence,
\[
\E\left(X^{\left(c\right)}+Y^{\left(c\right)}\right)
\to
\E\left(X+Y\right)
\]
as an element of \(\left[-\infty,\infty\right]\). The same argument, applied
to \(X'\) and \(Y'\), yields
\[
\E\left({X'}^{\left(c\right)}+{Y'}^{\left(c\right)}\right)
\to
\E\left(X'+Y'\right).
\]
But for every \(c>0\),
\[
\E\left(X^{\left(c\right)}+Y^{\left(c\right)}\right)
=
\E\left({X'}^{\left(c\right)}+{Y'}^{\left(c\right)}\right).
\]
Passing to the limit produces \(\E\left(X+Y\right)=\E\left(X'+Y'\right)\).\end{proof}

\begin{remark}
The lemma concerns the value of \(\E\left(X+Y\right)\) once that extended expectation is defined. As Simons points out, the existence of \(\E\left(X+Y\right)\) is not itself determined by the one-dimensional marginals alone \cite[p.~157]{Simons1977Unexpected}.
\end{remark}

\begin{lemma}[Upper bound]
For every \(n\ge2\), \(\sigma\left(n\right)\le\left\lceil n/2\right\rceil\).
\end{lemma}

\begin{proof}
Let \(m\coloneqq\left\lceil n/2\right\rceil\). Partition \(\left[n\right]\) into two sets \(A,B\) with \(\lvert A\rvert,\lvert B\rvert\le m\). For a random vector \(X=\left(X_1,\ldots,X_n\right)\), define \(U_X\coloneqq\sum_{i\in A}X_i\) and \(V_X\coloneqq\sum_{j\in B}X_j\), so that \(\sum_{i=1}^nX_i=U_X+V_X\).

Suppose \(X\) and \(Y\) have the same \(m\)-marginals and that both coordinate-sum expectations are defined. Since \(\lvert A\rvert,\lvert B\rvert\le m\), the marginal laws of \(\left(X_i\right)_{i\in A}\) and \(\left(Y_i\right)_{i\in A}\) agree, and similarly for \(B\). Hence, \(U_X\stackrel{d}=U_Y\) and \(V_X\stackrel{d}=V_Y\). By Simons's theorem, \(\E\left(U_X+V_X\right)=\E\left(U_Y+V_Y\right)\). Therefore, the expectation of the coordinate sum is determined by \(m\)-marginals.
\end{proof}

\subsection{The finite construction in odd dimension}

We begin by constructing the finite signed object used in the odd-dimensional counterexample. The construction has two stages. First, we build a finite signed exchangeable measure on zero-sum vectors with entries in (words over) a small alphabet. Its \(k\)-dimensional marginals will have a carefully chosen signed pattern. Later, we will convert this signed object into two probability laws with the same \(k\)-marginals but different expectations of their sums.

Fix an integer \(k\ge1\) and set \(n\coloneqq2k+1\). Let
\(A\coloneqq\left\{-2,-1,0,1,2\right\}\). For a count vector \(N=\left(N_{-2},N_{-1},N_0,N_1,N_2\right)\in\N^5\), define
\(\left\lvert N\right\rvert\coloneqq\sum_{a\in A}N_a\) and
\[\omega\left(N\right) \coloneqq \sum_{a\in A}aN_a = -2N_{-2}-N_{-1}+N_1+2N_2,\]
i.e., \(\left|N\right|\) is the number of letters in the word, while \(\omega(N)\) is the sum of its letters. Let \(\mathcal Z_n \coloneqq \left\{N\in\N^5\colon\left\lvert N\right\rvert=n,\omega\left(N\right)=0\right\}\). This is the set of count vectors for length-\(n\) words over \(A\) whose letters sum to zero.

We next prescribe the signed pattern we want to see in every \(k\)-coordinate marginal. For \(1\le r\le k\), let
\(m_r^{\left(2\right)}\coloneqq\left(0,0,k-r,0,r\right)\) be the
\(k\)-count type (that is, a count vector whose entries sum to \(k\)) with \(k-r\) zeros and \(r\) copies of \(2\), and let
\(m_r^{\left(1\right)}\coloneqq\left(0,0,k-r,r,0\right)\) be the
\(k\)-count type with \(k-r\) zeros and \(r\) copies of \(1\). Define
\[
d_r\coloneqq
\left(-1\right)^{r-1}
\frac{\binom{2k+1}{k-r}}{r\binom{2k+1}{k-1}}.
\]
Let \(q\) be the signed measure on \(k\)-count types given by
\(q\left(m_r^{\left(2\right)}\right)=d_r\) and
\(q\left(m_r^{\left(1\right)}\right)=-2d_r\) for \(1\le r\le k\), and
\(q\left(m\right)=0\) for all other \(k\)-count types \(m\).

The role of \(q\) is to encode the desired \(k\)-marginal behavior. The next lemma says that this local signed pattern is globally consistent: it can be realized by signed weights on \(\mathcal Z_n\), equivalently, by a signed exchangeable measure that assigns weight only to words whose letters sum to zero.

\begin{lemma}[Finite extension lemma]\label{lem:finex}
There exist signed weights \(\left(W_N\right)_{N\in\mathcal Z_n}\) such that, for every \(k\)-count type \(m\),
\[
q\left(m\right)=\sum_{N\in\mathcal Z_n}W_N\prod_{a\in A}\binom{N_a}{m_a}.
\]
\end{lemma}

\begin{proof}
Throughout this proof, for \(r\ge0\) we regard \(\binom{x}{r}\) as a polynomial in \(x\), with the convention \(\binom{x}{0}=1\).

Let \(\mathcal M_k\) denote the set of all \(k\)-count types over \(A\). Define the linear map \(K\colon\R^{\mathcal Z_n}\to\R^{\mathcal M_k}\) by
\[
\left(KW\right)\left(m\right)=\sum_{N\in\mathcal Z_n}W_N\prod_{a\in A}\binom{N_a}{m_a}.
\]
We need to show that \(q\in\operatorname{im}K\).

By finite-dimensional duality, \(\operatorname{im}K=\left(\ker K^*\right)^\perp\) with respect to the pairings
\[
\langle W,G\rangle\coloneqq\sum_{N\in Z_n}W_NG(N)
\qquad \text{and} \qquad
\langle q,\varphi\rangle\coloneqq\sum_{m\in M_k}q(m)\varphi(m).
\]
Consequently, it suffices to show that \(\langle q,\varphi\rangle=0\) for every \(\varphi\in\mathbb R^{M_k}\) with \(K^*\varphi=0\). For such a \(\varphi\), define
\[
P_\varphi(N)\coloneqq\sum_{m\in M_k}\varphi(m)\prod_{a\in A}\binom{N_a}{m_a}.
\]
The adjoint is given by \((K^*\varphi)(N)=P_\varphi(N)\), so the condition \(K^*\varphi=0\) says precisely that \(P_\varphi(N)=0\) for every \(N\in Z_n\).

We need a small interpolation fact. Let \(L\in\N\) and \(b\in\left\{-2,-1,0,1,2\right\}\). If a polynomial \(P\) of degree at most \(L\) vanishes on all integer points
\[
\left\{N\in\N^5\colon\lvert N\rvert=2L+1,\omega\left(N\right)=b\right\},
\]
then \(P\) vanishes identically on the affine space
\[
\left\{N\in\R^5\colon\lvert N\rvert=2L+1,\omega\left(N\right)=b\right\}.
\]

We prove this interpolation fact by induction on \(L\). For \(L=0\), \(P\) is constant, and the assertion is immediate. Assume \(L\ge1\). Let \(e_a\) denote the unit vector in the \(a\)-coordinate. Consider the directions \(u\coloneqq e_{-2}+e_2-2e_0\) and \(v\coloneqq e_{-1}+e_1-2e_0\). If \(b\ge0\), also use \(w\coloneqq e_{-1}+e_2-e_0-e_1\). If \(b<0\), use instead \(w'\coloneqq e_{-2}+e_1-e_{-1}-e_0\). The three chosen directions lie in the tangent space \(\left\{x\in\R^5\colon\lvert x\rvert=0,\omega\left(x\right)=0\right\}\), which has dimension \(3\). The chosen directions are linearly independent: for \(u,v,w\), the \(-2\)-coordinate forces the coefficient of \(u\) to vanish, then the \(2\)-coordinate forces the coefficient of \(w\) to vanish, and then the \(-1\)-coordinate forces the coefficient of \(v\) to vanish. The case \(u,v,w'\) is analogous, using the \(2\)-, \(-2\)-, and \(-1\)-coordinates. Therefore, the three chosen directions span the tangent space.

Each chosen direction has the form \(d=p-q\), where \(p,q\in\mathbb N^5\), \(|p|=|q|=2\), and \(\omega(p)=\omega(q)=s\). Explicitly,
\[
\begin{array}{c|c|c|c}
d&p&q&s\\
\hline
u=e_{-2}+e_2-2e_0&e_{-2}+e_2&2e_0&0\\
v=e_{-1}+e_1-2e_0&e_{-1}+e_1&2e_0&0\\
w=e_{-1}+e_2-e_0-e_1&e_{-1}+e_2&e_0+e_1&1\\
w'=e_{-2}+e_1-e_{-1}-e_0&e_{-2}+e_1&e_{-1}+e_0&-1
\end{array}
\]
For \(u\) and \(v\), \(s=0\), so \(b-s=b\). If \(b\ge0\), then \(b\in\{0,1,2\}\) and the third direction is \(w\), so \(b-s=b-1\in\{-1,0,1\}\). If \(b<0\), then \(b\in\{-2,-1\}\) and the third direction is \(w'\), so \(b-s=b+1\in\{-1,0\}\). Thus, \(b-s\in\{-2,-1,0,1,2\}\) in every case. In the next paragraph, fix one row of the table and use its \(d,p,q,s\).

Define \(Q\left(M\right)\coloneqq P\left(M+p\right)-P\left(M+q\right)\). Then \(Q\) has degree at most \(L-1\). Indeed, if \(P_L\) is the top homogeneous part of \(P\), then \(P_L\left(M+p\right)\) and \(P_L\left(M+q\right)\) have the same degree-\(L\) part, namely \(P_L\left(M\right)\), so the degree-\(L\) terms cancel.

If \(M\in\N^5\), \(\lvert M\rvert=2L-1\), and \(\omega\left(M\right)=b-s\), then both \(M+p\) and \(M+q\) lie in the integer slice \(\lvert N\rvert=2L+1\), \(\omega\left(N\right)=b\). Hence, \(Q\left(M\right)=0\) on that lower slice,
\[\left\{M\in\N^5\colon\lvert M\rvert=2L-1,\omega\left(M\right)=b-s\right\}.\]
By the induction hypothesis, \(Q\) vanishes identically on the whole real affine space \(\lvert M\rvert=2L-1\), \(\omega\left(M\right)=b-s\). Now if \(N\) lies in the real affine space \(\lvert N\rvert=2L+1\), \(\omega\left(N\right)=b\), then \(M=N-q\) lies in the preceding real affine space, 
\[\left\{M\in\R^5\colon\lvert M\rvert=2L-1,\omega\left(M\right)=b-s\right\}.\] Therefore, \(Q\left(N-q\right)=0\), i.e. \(P\left(N+d\right)=P\left(N\right)\) for all real \(N\) satisfying \(\lvert N\rvert=2L+1\) and \(\omega\left(N\right)=b\).

For a fixed \(N\) in this affine space and one of the chosen directions \(d\), set \(F\left(z\right)\coloneqq P\left(N+zd\right)\). Since \(d\) is tangent to the affine space, \(N+zd\) lies in the same affine space for every \(z\in\R\). Applying the identity \(P\left(M+d\right)=P\left(M\right)\) with \(M=N+zd\), we obtain \(F\left(z+1\right)=F\left(z\right)\) for every \(z\in\R\). Thus, \(F\) is a polynomial with period \(1\), and hence, \(F\) is constant. Therefore, \(P\left(N+zd\right)=P\left(N\right)\) for every \(z\in\R\). So \(P\) is invariant under every real multiple of each chosen direction. To see that this makes \(P\) constant on the affine slice, let \(N,N'\) lie in the slice. Then \(N'-N\) lies in the tangent space. Since the chosen directions span the tangent space, we may write
\[
N'-N=\alpha d_1+\beta d_2+\gamma d_3
\]
for the three chosen directions \(d_1,d_2,d_3\). Successively applying invariance along \(d_1,d_2,d_3\) delivers \(P(N')=P(N)\) and so \(P\) is constant on the affine slice.

The slice contains an integer point: if \(b\ne0\), take the count vector with \(N_b=1\), \(N_0=2L\), and all other counts zero; if \(b=0\), take the count vector with \(N_0=2L+1\) and all other counts zero. Since \(P\) vanishes on all integer points of the slice, the constant is \(0\). This proves the interpolation fact.

Returning to \(P_\varphi\), we have shown that \(P_\varphi\) vanishes identically on \(|N|=n,\omega(N)=0\). Let \(H_n\coloneqq\left\{N\in\mathbb R^5\colon |N|=n\right\}\). The restriction of the linear functional \(\omega\) to the affine space \(H_n\) is nonconstant; for instance,
\[
\omega\left(0,0,n,0,0\right)=0
\qquad \text{and} \qquad
\omega\left(0,0,n-1,1,0\right)=1.
\]
Choose affine coordinates \((z,h)\) on \(H_n\) with first coordinate \(z=\omega(N)\). In these coordinates the restricted polynomial has the form \(\tilde P(z,h)\) and satisfies \(\tilde P(0,h)=0\) for all \(h\). Viewing \(\tilde P\) as a polynomial in \(z\) whose coefficients are polynomials in \(h\), its constant coefficient is therefore, zero. Hence, \(\tilde P(z,h)=z\tilde R(z,h)\) for some polynomial \(\tilde R\). Equivalently, as functions on the hyperplane \(H_n\), there exists a polynomial \(R\) of degree at most \(k-1\) such that \(P_\varphi(N)=\omega(N)R(N)\).

Now consider the two lines \(N^{\left(2\right)}\left(s\right)\coloneqq\left(0,0,n-s,0,s\right)\) and \(N^{\left(1\right)}\left(s\right)\coloneqq\left(0,0,n-s,s,0\right)\). They satisfy \(\omega\left(N^{\left(2\right)}\left(s\right)\right)=2s\) and \(\omega\left(N^{\left(1\right)}\left(s\right)\right)=s\). Thus, \(P_\varphi\left(N^{\left(2\right)}\left(s\right)\right)=2sR\left(N^{\left(2\right)}\left(s\right)\right)\) and \(P_\varphi\left(N^{\left(1\right)}\left(s\right)\right)=sR\left(N^{\left(1\right)}\left(s\right)\right)\).

Let \(m_0\coloneqq\left(0,0,k,0,0\right)\) be the all-zero \(k\)-count type, and set \(f_0\coloneqq\varphi\left(m_0\right)\), \(f_r^{\left(2\right)}\coloneqq\varphi\left(m_r^{\left(2\right)}\right)\), and \(f_r^{\left(1\right)}\coloneqq\varphi\left(m_r^{\left(1\right)}\right)\) for \(1\le r\le k\). On the first line,
\[
P_\varphi\left(N^{\left(2\right)}\left(s\right)\right)=\sum_{r=0}^k\binom{s}{r}\binom{n-s}{k-r}f_r^{\left(2\right)},
\]
where \(f_0^{\left(2\right)}=f_0\). On the second line,
\[
P_\varphi\left(N^{\left(1\right)}\left(s\right)\right)=\sum_{r=0}^k\binom{s}{r}\binom{n-s}{k-r}f_r^{\left(1\right)},
\]
where \(f_0^{\left(1\right)}=f_0\). At \(N^{\left(0\right)}\coloneqq\left(0,0,n,0,0\right)\), all terms in \(P_\varphi\left(N^{\left(0\right)}\right)\) vanish except the all-zero \(k\)-count type, and, therefore, \(P_\varphi\left(N^{\left(0\right)}\right)=\binom{n}{k}f_0\). Since \(N^{\left(0\right)}\in\mathcal Z_n\), this value is \(0\). Hence, \(f_0=0\).

Using
\[
\left.\frac{d}{ds}\binom{s}{r}\right\rvert_{s=0}=\frac{\left(-1\right)^{r-1}}{r},\qquad (r\ge1),
\]
we obtain, for \(a\in\left\{1,2\right\}\),
\[
\left.\frac{d}{ds}P_\varphi\left(N^{\left(a\right)}\left(s\right)\right)\right\rvert_{s=0}=\sum_{r=1}^k\frac{\left(-1\right)^{r-1}}{r}\binom{n}{k-r}f_r^{\left(a\right)}.
\]
By the definition of \(d_r\),
\[
\sum_{r=1}^kd_rf_r^{\left(a\right)}=\frac{1}{\binom{n}{k-1}}\left.\frac{d}{ds}P_\varphi\left(N^{\left(a\right)}\left(s\right)\right)\right\rvert_{s=0}.
\]
Since both lines meet at \(N^{\left(0\right)}\), we get \(\left.\frac{d}{ds}P_\varphi\left(N^{\left(2\right)}\left(s\right)\right)\right\rvert_{s=0}=2R\left(N^{\left(0\right)}\right)\) and \(\left.\frac{d}{ds}P_\varphi\left(N^{\left(1\right)}\left(s\right)\right)\right\rvert_{s=0}=R\left(N^{\left(0\right)}\right)\). Therefore,
\[
\sum_{r=1}^kd_r\varphi\left(m_r^{\left(2\right)}\right)=\frac{2R\left(N^{\left(0\right)}\right)}{\binom{n}{k-1}}
\qquad \text{and} \qquad
\sum_{r=1}^kd_r\varphi\left(m_r^{\left(1\right)}\right)=\frac{R\left(N^{\left(0\right)}\right)}{\binom{n}{k-1}}.
\]
Consequently,
\[
\begin{aligned}
\left\langle q,\varphi\right\rangle&=\sum_{r=1}^kd_r\varphi\left(m_r^{\left(2\right)}\right)-2\sum_{r=1}^kd_r\varphi\left(m_r^{\left(1\right)}\right)\\
&=\frac{2R\left(N^{\left(0\right)}\right)}{\binom{n}{k-1}}-2\frac{R\left(N^{\left(0\right)}\right)}{\binom{n}{k-1}}=0.
\end{aligned}
\]
Thus, \(q\) annihilates \(\ker K^*\), and hence, \(q\in\operatorname{im}K\).
\end{proof}

\begin{lemma}\label{lem:anomaly}The coefficients \(d_r\) satisfy \(\sum_{r=1}^krd_r=\frac{k+2}{2k+1}\ne0\).
\end{lemma}

\begin{proof}
By definition,
\[
\sum_{r=1}^krd_r=\frac{1}{\binom{2k+1}{k-1}}\sum_{r=1}^k\left(-1\right)^{r-1}\binom{2k+1}{k-r}.
\]
Set \(j\coloneqq k-r\). Then
\[
\sum_{r=1}^k\left(-1\right)^{r-1}\binom{2k+1}{k-r}=\sum_{j=0}^{k-1}\left(-1\right)^{k-1-j}\binom{2k+1}{j}.
\]
The standard identity \(\sum_{j=0}^{m}\left(-1\right)^{m-j}\binom{N}{j}=\binom{N-1}{m}\) yields
\[
\sum_{j=0}^{k-1}\left(-1\right)^{k-1-j}\binom{2k+1}{j}=\binom{2k}{k-1}.
\]
Therefore,
\[
\sum_{r=1}^krd_r=\frac{\binom{2k}{k-1}}{\binom{2k+1}{k-1}}=\frac{k+2}{2k+1}. \qedhere
\]
\end{proof}

\subsection{From signed measures to probability laws}

We next explain how to turn a signed measure on the zero-sum hyperplane into two probability laws. Let \(c\coloneqq2/\pi\). Throughout this subsection we use the convention \(u\log\lvert u\rvert=0\) at \(u=0\). Define \(\Phi\left(u\right)\coloneqq\lvert u\rvert+icu\log\lvert u\rvert\). Thus, \(\Phi\left(0\right)=0\).

We use Nolan's \(S\left(\alpha,\beta,\gamma,\delta;1\right)\) parametrization of stable laws. By Nolan \cite[\S1.3, Definition~1.8, Eq.~(1.6)]{Nolan2020BasicProperties}, if \(Z\sim S\left(1,\beta,g,\delta;1\right)\), then
\[
\E e^{iuZ}=\exp\left\{-g\lvert u\rvert\left(1+i\beta\frac{2}{\pi}\sgn\left(u\right)\log\lvert u\rvert\right)+i\delta u\right\},
\]
with the convention \(0\log0=0\). Plugging in \(\beta=1\), \(g=\lambda\), and \(\delta=0\), and using \(\lvert u\rvert\sgn\left(u\right)=u\), produces
\[
\E e^{iuZ}=\exp\left\{-\lambda\left(\lvert u\rvert+\frac{2i}{\pi}u\log\lvert u\rvert\right)\right\}=\exp\left\{-\lambda\Phi\left(u\right)\right\}.
\]
Consequently, for every \(\lambda>0\), \(\exp\left\{-\lambda\Phi\left(u\right)\right\}\) is the characteristic function of a one-dimensional \(1\)-stable law.

Hence, if \(\Gamma\) is a finite positive measure on \(\R^n\) with finite support and \(\gamma\in\R^n\), then there is a probability law on \(\R^n\), \(\mu_{\Gamma,\gamma}\), whose characteristic function is
\[
\hat\mu_{\Gamma,\gamma}\left(t\right)=\exp\left(i\gamma\cdot t-\int_{\R^n}\Phi\left(t\cdot a\right)\,d\Gamma\left(a\right)\right).
\]
Indeed, if \(\Gamma=\sum_j\lambda_j\delta_{a_j}\), take independent one-dimensional random variables \(Z_j\) with characteristic functions \(\exp\left\{-\lambda_j\Phi\left(u\right)\right\}\). Then \(\gamma+\sum_ja_jZ_j\) has the displayed characteristic function.

\begin{proposition}\label{prop:stable} Let \(n\ge2\), and let \(H\coloneqq\left\{a\in\R^n\colon\sum_{i=1}^na_i=0\right\}\). Let \(\Lambda\) be a finitely supported signed measure on \(H\). Write \(\Lambda=\Lambda^+-\Lambda^-\) for its Jordan decomposition, and understand all integrals against \(\Lambda\) as signed integrals.\footnote{Since \(\Lambda\) has finite support, no integrability issue arises.}

Fix \(1\le k<n\). Suppose there exists \(\ell=\left(\ell_1,\ldots,\ell_n\right)\in\R^n\) such that, for every \(T\subseteq\left[n\right]\) with \(\lvert T\rvert=k\) and every \(t\in\R^T\),
\[
\int_H\left\lvert t\cdot a_T\right\rvert\,d\Lambda\left(a\right)=0 \qquad \text{and} \qquad
\int_H\left(t\cdot a_T\right)\log\left\lvert t\cdot a_T\right\rvert\,d\Lambda\left(a\right)=t\cdot\ell_T.
\]
If \(\ell\cdot\1\ne0\), then there exist random vectors \(X^+,X^-\in\R^n\) such that \(\left(X_i^+\right)_{i\in T}\stackrel{d}=\left(X_i^-\right)_{i\in T}\) for every \(T\subseteq\left[n\right]\) with \(\lvert T\rvert=k\), but
\[
\E\sum_{i=1}^nX_i^+\ne\E\sum_{i=1}^nX_i^-.
\]
Moreover, both sums are deterministic constants.
\end{proposition}

\begin{proof}
Let \(c\coloneqq2/\pi\). Choose location vectors \(\gamma^+,\gamma^-\in\R^n\) satisfying \(\gamma^+-\gamma^-=c\ell\). Let \(X^+\sim\mu_{\Lambda^+,\gamma^+}\) and \(X^-\sim\mu_{\Lambda^-,\gamma^-}\).

Fix \(T\subseteq\left[n\right]\) with \(\lvert T\rvert=k\). For \(t\in\R^T\), the characteristic functions of the \(T\)-marginals have exponents
\[
i\gamma_T^+\cdot t-\int_H\Phi\left(t\cdot a_T\right)\,d\Lambda^+\left(a\right) \quad \text{and} \quad
i\gamma_T^-\cdot t-\int_H\Phi\left(t\cdot a_T\right)\,d\Lambda^-\left(a\right).
\]
Subtracting the second exponent from the first produces
\[
i\left(\gamma_T^+-\gamma_T^-\right)\cdot t-\int_H\Phi\left(t\cdot a_T\right)\,d\Lambda\left(a\right).
\]
By assumption, \(\int_H\Phi\left(t\cdot a_T\right)\,d\Lambda\left(a\right)=ic\,t\cdot\ell_T\). Since \(\gamma^+-\gamma^-=c\ell\), the exponent difference is \(0\). Hence, the \(T\)-marginal characteristic functions agree, and therefore, \(\left(X_i^+\right)_{i\in T}\stackrel{d}=\left(X_i^-\right)_{i\in T}\).

Now let \(\1=\left(1,\ldots,1\right)\). Since \(\Lambda^\pm\) are supported on \(H\), \(a\cdot\1=0\) for every \(a\in\operatorname{supp}\Lambda^\pm\). Therefore, for \(s\in\R\),
\[
\E\exp\left(is\sum_{i=1}^nX_i^\pm\right)=\exp\left(is\gamma^\pm\cdot\1\right).
\]
The right side is the characteristic function of the point mass at \(\gamma^\pm\cdot\mathbf 1\). By uniqueness of characteristic functions,
\(\sum_{i=1}^nX_i^\pm=\gamma^\pm\cdot\mathbf 1\)
a.s. In particular, the two expectations are finite and
\[
\E\sum_{i=1}^nX_i^+-\E\sum_{i=1}^nX_i^-=\left(\gamma^+-\gamma^-\right)\cdot\1=c\ell\cdot\1,
\]
which is nonzero by assumption.
\end{proof}

\subsection{The lower bound in all dimensions}

\begin{lemma}[Odd-dimensional lower bound]
For every \(k\ge1\), \(\sigma\left(2k+1\right)>k\).
\end{lemma}

\begin{proof}
Set \(n\coloneqq2k+1\). By the finite extension lemma (\ref{lem:finex}), choose signed weights \(\left(W_N\right)_{N\in\mathcal Z_n}\) such that \(q\left(m\right)=\sum_{N\in\mathcal Z_n}W_N\prod_{a\in A}\binom{N_a}{m_a}\).

For \(N\in\mathcal Z_n\), let \(\left[N\right]\) denote the uniform probability measure on all vectors \(a\in A^n\) with count vector \(N\). Define the signed measure
\[
\Lambda\coloneqq\binom{n}{k}\sum_{N\in\mathcal Z_n}W_N\left[N\right].
\]
Since \(N\in\mathcal Z_n\) implies \(\omega\left(N\right)=0\), the measure \(\Lambda\) is supported on \(H=\left\{a\in\R^n\colon\sum_{i=1}^na_i=0\right\}\).

Fix \(T\subseteq\left[n\right]\) with \(\lvert T\rvert=k\). We first compute the \(T\)-projection of \(\Lambda\). Under \(\left[N\right]\), the projected count type \(m\) on \(T\) has hypergeometric mass \(\binom{n}{k}^{-1}\prod_{a\in A}\binom{N_a}{m_a}\). Conditional on this count type, the arrangements on \(T\) are uniform. Hence, after multiplying by the prefactor \(\binom{n}{k}\) in the definition of \(\Lambda\), each arrangement of type \(m\) has signed mass \(q\left(m\right)/\#\left\{\text{arrangements of type }m\right\}\). For the two nonzero families of count types, the number of arrangements is \(\binom{k}{r}\). Therefore, the pushforward of \(\Lambda\) under the projection map \(\pi_T\)
\[
\left(\pi_T\right)_\#\Lambda=\sum_{r=1}^k\frac{d_r}{\binom{k}{r}}\sum_{\substack{R\subseteq T\\\lvert R\rvert=r}}\left(\delta_{2\1_R}-2\delta_{\1_R}\right),
\]
where \(\1_R\in\R^T\) is the indicator vector of \(R\).

Let \(t\in\R^T\) and write \(t_R\coloneqq\sum_{i\in R}t_i\). For each \(R\), \(\lvert2t_R\rvert-2\lvert t_R\rvert=0\). Therefore, \(\int_H\lvert t\cdot a_T\rvert\,d\Lambda\left(a\right)=0\). For the logarithmic term (that is, for the integral of \(\left(t\cdot a_T\right)\log\lvert t\cdot a_T\rvert\) against \(\Lambda\)), each pair \(\delta_{2\1_R}-2\delta_{\1_R}\) contributes
\[
\left(2t_R\right)\log\left\lvert2t_R\right\rvert-2t_R\log\left\lvert t_R\right\rvert=2t_R\log2,
\]
with the convention \(u\log\lvert u\rvert=0\) at \(u=0\). Thus,
\[\int_H\left(t\cdot a_T\right)\log\left\lvert t\cdot a_T\right\rvert\,d\Lambda\left(a\right) =2\log2\sum_{r=1}^k\frac{d_r}{\binom{k}{r}}\sum_{\substack{R\subseteq T\\\lvert R\rvert=r}}t_R.\]
Since \(\sum_{\substack{R\subseteq T\\\lvert R\rvert=r}}t_R=\binom{k-1}{r-1}\sum_{i\in T}t_i\), we get
\[
\int_H\left(t\cdot a_T\right)\log\left\lvert t\cdot a_T\right\rvert\,d\Lambda\left(a\right)=\lambda_k\sum_{i\in T}t_i,
\]
where
\[
\lambda_k\coloneqq2\log2\sum_{r=1}^k\frac{d_r}{\binom{k}{r}}\binom{k-1}{r-1}.
\]
Using \(\binom{k-1}{r-1}/\binom{k}{r}=r/k\), we find
\[
\lambda_k=\frac{2\log2}{k}\sum_{r=1}^krd_r.
\]
By the anomaly lemma (\ref{lem:anomaly}), \(\sum_{r=1}^krd_r=\frac{k+2}{2k+1}\). Hence,
\[
\lambda_k=\frac{2\left(k+2\right)}{k\left(2k+1\right)}\log2\ne0.
\]

Consequently, the hypotheses of the stable realization proposition (\ref{prop:stable}) hold with \(\ell=\lambda_k\left(1,\ldots,1\right)\). Since \(\ell\cdot\1=\left(2k+1\right)\lambda_k=\frac{2\left(k+2\right)}{k}\log2\ne0\), there exist random vectors \(X^+,X^-\in\R^{2k+1}\) with identical \(k\)-dimensional marginals but with \(\E\sum_{i=1}^{2k+1}X_i^+\ne\E\sum_{i=1}^{2k+1}X_i^-\). Equality of all \(k\)-dimensional marginals implies equality of all lower-dimensional marginals. Therefore, \(k\)-marginals do not determine the extended expectation of the coordinate sum in dimension \(2k+1\).
\end{proof}

\begin{lemma}[Even-dimensional lower bound]
For every \(k\ge1\), \(\sigma\left(2k+2\right)>k\).
\end{lemma}

\begin{proof}
By the odd-dimensional lower bound, there exist random vectors \(X^+,X^-\in\R^{2k+1}\) with the same \(k\)-marginals but with different expectations of their coordinate sums. Define \(\tilde X^+\coloneqq\left(X_1^+,\ldots,X_{2k+1}^+,0\right)\) and \(\tilde X^-\coloneqq\left(X_1^-,\ldots,X_{2k+1}^-,0\right)\). Then \(\tilde X^+\) and \(\tilde X^-\) have the same \(k\)-marginals in dimension \(2k+2\). Indeed, if a \(k\)-coordinate set does not include the last coordinate, this follows directly from the construction. If it does include the last coordinate and \(k\ge2\), then the remaining \(k-1\) original coordinates have equal laws because the original vectors have equal lower-dimensional marginals, obtained by marginalizing their \(k\)-marginals. If \(k=1\), the selected coordinate is just the added deterministic coordinate \(0\) for both vectors. Their coordinate sums differ by the same amount as before. Therefore, \(k\)-marginals do not determine the extended expectation of the coordinate sum in dimension \(2k+2\).
\end{proof}

\begin{proof}[Proof of the main theorem]
For every \(n\ge2\), we have the upper bound \(\sigma\left(n\right)\le\left\lceil n/2\right\rceil\). If \(n=2\), by Simons's theorem, \(\sigma\left(2\right)\le1\); while by definition \(\sigma\left(2\right)\ge1\). Hence, \(\sigma\left(2\right)=1\).

Now let \(n=2k+1\) with \(k\ge1\). By the odd-dimensional lower bound, \(\sigma\left(2k+1\right)>k\), while the upper bound is \(\sigma\left(2k+1\right)\le k+1\). Therefore, \(\sigma\left(2k+1\right)=k+1=\left\lceil\left(2k+1\right)/2\right\rceil\).

Finally, let \(n=2k+2\) with \(k\ge1\). By the even-dimensional lower bound, \(\sigma\left(2k+2\right)>k\), while the upper bound is \(\sigma\left(2k+2\right)\le k+1\). Therefore, \(\sigma\left(2k+2\right)=k+1=\left\lceil\left(2k+2\right)/2\right\rceil\). Thus, for every \(n\ge2\), \(\sigma\left(n\right)=\left\lceil n/2\right\rceil\).
\end{proof}

\bibliographystyle{amsplain}
\bibliography{references}

\end{document}